\documentclass[12pt,reqno,a4paper]{article}

\usepackage[usenames]{color}

\definecolor{webgreen}{rgb}{0,.5,0}
\definecolor{webbrown}{rgb}{.6,0,0}
\definecolor{red}{rgb}{1,0,0}

\usepackage{amssymb, float, graphics,amsmath, amsfonts,latexsym,epsf,amsthm}
\usepackage[colorlinks=true, linkcolor=webgreen, filecolor=webbrown, citecolor=red]{hyperref}

\newcommand{\bb}{\begin{equation}}
\newcommand{\ee}{\end{equation}}

\renewcommand{\H}{\mathcal{H}}
\newcommand{\HH}{\mathbf{H}}
\newcommand{\BB}{\mathbf{B}}

\newcommand{\Bf}{\mathcal{B}}

\newcommand{\N}{{\mathbb N}}
\newcommand{\R}{{\mathbb R}}
\newcommand{\A}{{\mathcal A}}

\newcommand{\seqn}[1]{\left(#1\right)_{n \in \N_0}}
\newcommand{\supp}{\mathrm{supp}}
\renewcommand{\(}{\left(}
\renewcommand{\)}{\right)}
\renewcommand{\[}{\left[}
\renewcommand{\]}{\right]}

\newcommand{\seqnum}[1]{\href{
http://www.research.att.com/cgi-bin/access.cgi/as/~njas/sequences/eisA.cgi?Anum=#1}{\underline{#1}}}

 \newtheorem{thm}{Theorem}[section]
 \newtheorem{prop}[thm]{Proposition}
 \newtheorem{lem}[thm]{Lemma}

 \newtheorem{dfn}{Definition}[section]
 
\theoremstyle{definition}

\title{\bf Hankel transform of a sequence obtained by series reversion II - aerating transforms}
\author{\frenchspacing
{\bf Radica Boji\v ci\'c}\\
{\small\it University of Pri\v stina, Faculty of Economy, Serbia}\\
{\small E-mail:\  tallesboj@gmail.com}\\
[2mm]
{\bf Marko D. Petkovi\'c}\\
{\small\it University of Ni\v s, Faculty of Sciences and Mathematics, Serbia}\\
{\small E-mail:\ dexterofnis@gmail.com} \\
[2mm]
{\bf Paul Barry} \\
{\small\it School of Science, Waterford Institute of Technology, Ireland}\\
{\small E-mail: pbarry@wit.ie}
}

\date{}

\begin{document}

\maketitle

\begin{abstract}

This paper provides the connection between the Hankel transform and aerating transforms of a given integer sequence. Results obtained are used to establish a completely different Hankel transform evaluation of the series reversion of a certain rational function $Q(x)$ and shifted sequences, recently published in our paper \cite{part1}. For that purpose, we needed to evaluate the Hankel transforms of the sequences $\seqn{\alpha^2 C_n-\beta C_{n+1}}$ and $\seqn{\alpha^2 C_{n+1}-\beta C_{n+2}}$, where $C=\seqn{C_n}$ is the well-known sequence of Catalan numbers. This generalizes the results of Cvetkovi\' c, Rajkovi\'c and Ivkovi\'c \cite{CRI}. Also, we need the evaluation of  Hankel-like determinants whose entries are Catalan numbers $C_n$ and which is based on the recent results of Krattenthaler \cite{krattCat}. The results obtained are general and can be applied to many other Hankel transform evaluations.
\end{abstract}


\smallskip

\section{Introduction}

The Hankel transform of a given sequence $a=\seqn{a_n}$ is defined as the sequence $h=\seqn{h_n}$ of Hankel determinants, i.e.
\bb
{h_n}=\det\([a_{i+j}]_{0\leq i,j \leq n}\), \quad (n\in \N_0)
\label{hank}
\ee
and denoted by $h=\H(a)$. The term ``Hankel transform'' was first introduced by Layman \cite{layman} in 2001. Despite that, many Hankel determinants evaluation were obtained much earlier, mostly due to their important combinatorial properties (see for example \cite{Brualdi,Gessel-Xin,Sulanke,Viennot}).

\smallskip

Papers \cite{Brualdi,CRI,RPB} use a method based on orthogonal polynomials (or continued fractions) to provide a Hankel transform evaluation of different sequences. It is also used in our recently published paper \cite{part1} where we evaluated the Hankel transform of a series reversion of the function $\frac {x}{1+\alpha x+\beta x^{2}}$, as well as of the corresponding shifted sequences.

\smallskip

In this paper, we also show another evaluation for the same sequences, which is based on the application of the falling $\alpha$-binomial transform \cite{spivey} and aerating transforms. The method described in this paper provides us with more general results regarding the connection between Hankel transforms and aerating transforms.

\smallskip

For our purpose we need the Hankel transform evaluation of $\seqn{\alpha^2 C_n-\beta C_{n+1}}$ and $\seqn{\alpha^2 C_{n+1}-\beta C_{n+2}}$, where $C=\seqn{C_n}$ is the well-known sequence of Catalan numbers. This generalizes results of Cvetkovi\' c, Rajkovi\'c and Ivkovi\'c \cite{CRI}. We also need the evaluation of  Hankel-like determinants whose entries are Catalan numbers $C_n$ and which is based on the recent results of Krattenthaler \cite{krattCat}.

\smallskip

For our further discussion we need to recall the definition of the series reversion of a (generating) function $f(x)$ which satisfies $f(0)=0$ (see \cite{PB}).

\begin{dfn} For a given (generating) function $v=f(u)$ with the property $f(0)=0$, the
series reversion is the sequence $\seqn{s_n}$ such that
$$
u=f^{-1}(v)=s_1v+s_2v^2+\cdots+s_nv^n+\cdots,
$$
where $u=f^{-1}(v)$ is the inverse function of $v=f(u)$. Note that since $f(0)=0$, there must hold $s_0=f^{-1}(0)=0$.
\label{dfn:series}
\end{dfn}

\smallskip

\section{The series reversion of $\frac {x}{1+\alpha x+\beta x^{2}}$}

We recall a few basic properties and expressions of the sequence $\seqn{u_n}$ obtained by reverting
$$
Q(x)=\frac{x}{1+\alpha x+\beta x^{2}}.
$$
This sequence is already investigated in \cite{PB, part1}. The generating function $U(x)$ satisfies $Q(U(x))=x$ (Definition \ref{dfn:series}) and is given by
\bb
U(x) =
\frac {1 - \alpha x - \sqrt{1 - 2 \alpha x + (\alpha^{2} - 4
\beta)x^{2}}}{2 \beta x}.
\label{eq:formulaU}
\ee
The general term of the sequence $\seqn{u_n}$ can be expressed in the following way (Proposition 9 in \cite{PB}):
\bb
u_n = \sum_{k=0}^{[\frac{n-1}{2}]} \binom {n-1}{2k}C_k\alpha^{n-2k-1}\beta^{k}.
\label{formulaUN}
\ee
Note that the sequence $\seqn{u_n}$ generalizes the sequence $\seqn{C_n+\delta_{n0}}$ (for $\alpha=2$ and $\beta=1$) and more generally the sequence $\seqn{(N_n(z)-\delta_{n0})/z}$ where $N_n(z)$ is the $n$-th Narayana polynomial (for $\alpha=z+1$ and $\beta=z$).

\smallskip

Consider the shifted sequences $\seqn{u^*_n}$ and $\seqn{u^{**}_n}$ defined by $u_n^*=u_{n+1}$ and $u^{**}_n=u_{n+2}$. Also denote by $h_n$, $h^*_n$ and $h^{**}_n$, the Hankel transforms of the sequences $u_n$, $u^*_n$ and $u^{**}_n$ respectively. Our previous paper \cite{part1} provides the evaluation of $h_n^*$, $h_n^{**}$ and $h_n$ using the method based on orthogonal polynomials \cite{CRI,RPB}. The main results are the following theorems (Theorem 4.3, Theorem 4.4, and Corollary 5.4 in \cite{part1}):

\begin{thm} {\bf \cite{part1}}
The Hankel transform of the sequence $\seqn{u^*_n}$ is given by
\bb
h^*_n=\beta^{\binom{n+1}{2}}.
\label{eq:h*}
\ee
\label{thm:h_n*}
\end{thm}

\begin{thm} {\bf \cite{part1}}
The Hankel transform of the sequence $\seqn{u^{**}_n}$ is given by
\bb
h^{**}_n=\frac{\beta^{\binom{n+1}{2}}}{2^{n+1}\sqrt{\alpha^2-4\beta}}\big[(\alpha+\sqrt{\alpha^2-4\beta})^{n+2}-
(\alpha-\sqrt{\alpha^2-4\beta})^{n+2} \big].
\label{eq:h_n**}
\ee
\label{thm:h_n**}
\end{thm}

\begin{thm} {\bf \cite{part1}}
The Hankel transform of the sequence $\seqn{u_n}$ is given by
\bb
h_n=\frac{\beta^{\binom{n}2}}{2^n \sqrt{\alpha^2-4\beta}} \[\(\alpha-\sqrt{\alpha^2-4\beta}\)^n-\(\alpha+\sqrt{\alpha^2-4\beta}\)^n \].
\label{formzahn}
\ee
\label{thm:h_n}
\end{thm}

Note that the sequence $\seqn{u^*_n}$ reduces to the sequence $\seqn{C^a_n}$ of {\it aerated Catalan numbers} (\seqnum{A126120}), by choosing $\alpha=0$ and $\beta=1$ (yields directly from \eqref{eq:formulaU}). The sequence $\seqn{C^a_n}$ is defined by
$$
C^a_n=\begin{cases}
C_{n/2},& n\textrm{ is even} \\
0,& n\textrm{ is odd}
\end{cases}.
$$
Also note that the Hankel transform of the aerated sequence $\seqn{C^a_n}$ is $\seqn{1}$ (which is proven in the section \ref{sect:dnun*}), the same as in the case of the Catalan sequence (see for example \cite{krattCat}). That result raises the more general question about the Hankel transform of {\it aerated sequences}, which we deal with in the rest of the paper.

%Using the well-known Hankel transform evaluation of Catalan and shifted Catalan numbers, we provide an alternative proofs of the Theorems \ref{thm:h_n*}-\ref{thm:h_n}.
%\smallskip
%We also evaluate a few determinants which has similar structure as Hankel determinants. We call them {\it Hankel-like determinants}.

%\section{Hankel transform of aerated sequences and related Hankel-like determinants}
%\label{sect:3}

\section{The falling $\alpha$-binomial transform}

The following transform is a generalization of the well-known binomial transform and was introduced by Spivey and Steil \cite{spivey}. We will use it in further considerations.

\begin{dfn} For a given sequence $a=\seqn{a_n}$, its falling $\alpha$-binomial transformation $b=\Bf(a;\alpha)$ is defined by
$$
b_n=\sum_{k=0}^n \binom nk \alpha^{n-k} a_k.
$$
\end{dfn}

Spivey and Steil \cite{spivey} proved that the Hankel transform is invariant under the falling $\alpha$-binomial transform for arbitrary $\alpha$. In other words, the following lemma is valid.

\begin{lem} {\bf \cite{spivey}} For an arbitrary sequence $a=\seqn{a_n}$ and number $\alpha$, it holds that $\H(\Bf(a;\alpha))=\H(a)$.
\label{lem:Bf}
\end{lem}

The falling $\alpha$-binomial transform can be written in the following matrix form
$$
b=\BB^\alpha a, \quad \BB^\alpha=\[ \binom nk \alpha^{n-k} \]_{n,k\in \N_0}
$$
where we treat the sequences $a$ and $b$ as corresponding column vectors (we also use this notation in the rest of the paper). We call the matrix $\BB^\alpha$ the {\it $\alpha$-binomial matrix}. The following lemma shows the connection between the Hankel matrices
$$
\HH_a=[a_{i+j}]_{i,j\in \N_0}, \quad \HH_b=[b_{i+j}]_{i,j\in \N_0}
$$
and the matrix $\BB^\alpha$.

\begin{lem}
\label{lem:binmatrix}
If $b=\Bf(a;\alpha)$ then there holds
\bb
\HH_b=\BB^\alpha \HH_a (\BB^\alpha)^T.
\ee
\end{lem}
\begin{proof} Let us start from the general element $b_{n+m}$ of the matrix $\HH_b$:
$$
b_{n+m}=\sum_{t} \binom{n+m}t \alpha^{n+m-t} a_t.
$$
Using the well-known identity
$$
\binom{n+m}t=\sum_{k=0}^{n} \binom nk \binom m{t-k}
$$
we obtain
$$
\aligned
b_{n+m}&=\sum_{t=0}^{n+m} \sum_{k=0}^{n} \binom nk \binom m{t-k} \alpha^{n+m-t} a_t
=\sum_{l=0}^m \sum_{k=0}^n \binom nk \binom ml \alpha^{n+m-k-l}a_{k+l}\\
&=\sum_{l=0}^m \sum_{k=0}^n \binom nk \alpha^{n-k} \cdot a_{k+l} \cdot \binom ml \alpha^{m-l}
=\sum_{l=0}^m \sum_{k=0}^n \(\BB^\alpha\)_{nk} \cdot a_{k+l} \cdot \(\BB^\alpha\)_{ml}.
\endaligned
$$
This completes the proof of the lemma.
\end{proof}

In the following sections, we give alternative proofs (to those given in \cite{part1}) of theorems regarding the Hankel transform of $u_n^*$, $u_n^{**}$ and $u_n$ (Theorem \ref{thm:h_n*}, Theorem \ref{thm:h_n**} and Theorem \ref{thm:h_n}).

\section{The sequence $u_n^*$}
\label{sect:dnun*}

We give an evaluation of the Hankel transform of $u_n^*$ just using transformations and known results concerning the Hankel transform of the Catalan numbers \cite{krattCat}. For this purpose, we define the following {\it aerating transform}.

\begin{dfn} For a given sequence $c=\seqn{c_n}$, we define its aerating transformation $p=\A(c)$ by
$$
p_n=\begin{cases}
c_{n/2},& n\textrm{ is even} \\
0,& n\textrm{ is odd}
\end{cases}.
$$
In other words, if $p=\A(c)$ then $p=(c_0,0,c_1,0,c_2,0,c_3,0,\ldots)$.
\end{dfn}

Hence $C^a=\A(C)$ where $C=\seqn{C_n}$ is the sequence of Catalan numbers. The following theorem shows the connection between the Hankel transform of a given sequence $c$ and its aerated sequence $p=\A(c)$.

\begin{thm} Let $g=\H(p)$ and $h=\H(c)$ where $p=\A(c)$ is the aerated sequence of $c$. Then there holds
$$
\det[p_{i+j}]_{0\leq i,j\leq n}=\det[c_{i+j}]_{0 \leq i,j \leq \lfloor \frac{n}2 \rfloor} \cdot \det[c_{i+j+1}]_{0 \leq i,j \leq \lfloor \frac{n-1}2 \rfloor}.
$$
In terms of Hankel transforms, this last equality can be written as
$$
g_n=h_{\lfloor \frac{n}2 \rfloor} h^*_{\lfloor \frac{n-1}2 \rfloor},
$$
where $h^*$ is the Hankel transform of the shifted sequence $c^*=\seqn{c_{n+1}}$ and $h^*=\H(c^*)$.
\label{thm:aerated}
\end{thm}
\begin{proof}

%First assume that $n$ is even, i.e. $n=2k$. Let us start from the determinant
%$$
%\det(a_{i+j})_{0\leq i,j\leq n-1}=
%\left|
%\begin{array}{ccccccc}
%c_0 & 0 & c_1 & 0 & c_2 & \cdots & c_{k-1}\\
%0 & c_1 & 0 & c_2 & 0 &   &  0 \\
%c_1 & 0 & c_2 & 0 & c_3 &  &  c_k \\
%0 & c_2 & 0 & c_3 & 0 &  & 0 \\
%c_2 & 0 & c_3 & 0 & c_4 &  & c_{k+1} \\
%\vdots & & & & & \ddots & \vdots \\
%c_{k-1} & 0 & c_k & 0 & c_{k+1} & \cdots  &  c_{2k-1} \\
%\end{array}
%\right|
%$$
By exchanging the rows and columns of the determinant $\det[p_{i+j}]_{0\leq i,j\leq n-1}$ we obtain
$$
\det[p_{i+j}]_{0\leq i,j\leq n}=
\left|
\begin{array}{cccccc}
c_0 & 0 & c_1 & 0 & c_2 & \cdots \\
0 & c_1 & 0 & c_2 & 0 &   \\
c_1 & 0 & c_2 & 0 & c_3 &   \\
0 & c_2 & 0 & c_3 & 0 &   \\
c_2 & 0 & c_3 & 0 & c_4 &  \\
\vdots & & & & & \ddots
\end{array}
\right|
=\det \bmatrix {\bf A} &  \\  & {\bf B} \endbmatrix
$$
where ${\bf A}=[c_{i+j}]_{0 \leq i,j \leq \lfloor \frac{n}2 \rfloor-1}$ and ${\bf B}=[c_{i+j+1}]_{0 \leq i,j \leq \lfloor \frac{n-1}2 \rfloor-1}$. Now the statement of the theorem follows immediately.
\end{proof}

It is known (see for example \cite{kratt}) that $\H\(\seqn{C_n}\)=\H\(\seqn{C_{n+1}}\)=\seqn{1}$. Using Theorem \ref{thm:aerated} (for $c_n=C_n$) we obtain the result $\H\(\seqn{C^a_n}\)=\seqn{1}$.

We also need the following proposition:

\begin{prop} Let $c=\seqn{c_n}$ be an arbitrary sequence and $h=\H(c)$ its Hankel transform. Then $\H\( \seqn{r^n c_n}\)=\seqn{r^{n(n+1)}h_n}$ where $r$ is an arbitrary number.
\label{prop:mul}
\end{prop}

\medskip

Let $c_n=\beta^n C_n$ and let $p=\A(c)$. Recall that $u_n^*$ can be expressed as follows (directly from \eqref{formulaUN}):
$$
u^*_n=u_{n+1}=\sum_{k=0}^{[\frac{n}{2}]} \binom {n}{2k}\alpha^{n-2k}\beta^{k}C_k=\sum_{l=0}^n \binom nl \alpha^{n-l}p_l
$$
which implies that $\seqn{u^*_n}=\Bf\(p;\alpha\)$.

Note that the sequence $p=\A(c)$ can be expressed as $p=\seqn{\beta^{n/2}C^a_n}$. Now using Lemma \ref{lem:Bf} and Proposition \ref{prop:mul} we obtain the result of Theorem \ref{thm:h_n*}:
$$
h^*=\H(u^*)=\H\(p\)=\H\(\seqn{\beta^{n/2}C^a_n} \)=\seqn{\beta^{\binom{n+1}2}}.
$$

\section{The Hankel transform of a linear combination of Catalan and shifted Catalan numbers}

Cvetkovi\'c, Rajkovi\' c and Ivkovic \cite{CRI} considered the Hankel transform of the sequence $\seqn{C_n+C_{n+1}}$. In this section, we generalize their result, providing the Hankel transform evaluation of the sequences $\seqn{\alpha^2 C_n-\beta C_{n+1}}$ and $\seqn{\alpha^2 C_{n+1}-\beta C_{n+2}}$. We also need this result in the next section.

\smallskip

We use the method based on orthogonal polynomials, as used in \cite{CRI,RPB}. It assumes that a given sequence $\seqn{a_n}$ is a {\it moment sequence} with respect to some weight function (measure) $w(x)$, i.e. that there holds
$$
a_n=\int_\R x^n w(x)dx.
$$
If $h=\H(a)$ and $h_n\neq 0$ for all $n\in \N_0$, then there exists a sequence of monic orthogonal polynomials $\seqn{\pi_n(x)}$, which satisfies a three-term recurrence relation
\bb
\pi_{n+1}(x)=(x-\alpha_n)\pi_n(x)-\beta_n\pi_{n-1}(x).
\ee
The Hankel transform $h_n$ can be evaluated using the following Heilermann formula (see for example \cite{kratt}):
\bb
\label{formula}
h_n=a_0^{n+1} \beta_1^{n} \beta_2^{n-1} \cdots \beta_{n-1}^2 \beta_{n}.
\ee
In order to establish closed-form expression for coefficients $\alpha_n$ and $\beta_n$, the following transformation lemmas can be useful:

\begin{lem} {\bf \cite{part1}}
Let $w(x)$ and $\tilde w(x)$ be weight functions and denote by $\seqn{\pi_n(x)}$ and $\seqn{\tilde \pi_n(x)}$ the corresponding
orthogonal polynomials. Also denote by $\seqn{\alpha_n},\seqn{\beta_n}$ and $\seqn{\tilde\alpha_n},\seqn{\tilde \beta_n}$ the three-term relation coefficients corresponding to $w(x)$ and $\tilde w(x)$ respectively. The following transformation formulas are valid:
\begin{itemize}
\item[{\bf (1)}] If $\tilde w(x)=Cw(x)$ where $C>0$ then we have
$\tilde \alpha_n=\alpha_n$ for $n \in \N_0$ and
$\tilde\beta_0=C\beta_0$, $\tilde \beta_n=\beta_n$ for $n \in \N$.
Additionally there holds $\tilde \pi_n(x)=\pi_n(x)$ for all $n \in
\N_0$. \item[{\bf (2)}] If $\tilde w(x)=w(ax+b)$ where $a,b \in
\R$ and $a\neq 0$ there holds $\tilde\alpha_n =
\frac{\alpha_n-b}{a}$ for $n \in \N_0$ and $\tilde\beta_0 =
\frac{\beta_0}{|a|}$ and $\tilde\beta_n  = \frac{\beta_n}{a^2}$
for $n \in \N$. Additionally there holds $\tilde
\pi_n(x)=\frac1{a^n}\pi_n (ax+b)$.
\end{itemize}
\label{lem:ModWeight1}
\end{lem}

\begin{lem} {\bf (Linear multiplier transformation) \cite{gautschi}} Consider the same notation as in Lemma \ref{lem:ModWeight1}.
Let the sequence $\seqn{r_n}$ be defined by
\bb
r_0=c-\alpha_0,\qquad r_n=c-\alpha_n- \frac{\beta_n}{r_{n-1}}
\qquad (n \in \N_0).
\label{eq:r2}
\ee
If $\tilde w(x)=(x-c)w(x)$
where $c< \inf \supp(w),$ then there holds
\bb
\aligned
\tilde\beta_0&=\int_{\R}\tilde w(x)\ dx, \quad \tilde\beta_n=\beta_n\frac{r_n}{r_{n-1}}, \qquad (n \in \N), \\
\tilde\alpha_n&=\alpha_{n+1}+r_{n+1}-r_n, \quad (n \in \N_0).\
\endaligned
\label{eq:xcm2}
\ee
\label{lem:ModWeight3}
\end{lem}

Now we prove the main theorem of this section. The proof is based on the sequential application of the previous two lemmas for a known weight function, i.e. a weight function whose coefficients $\alpha_n$ and $\beta_n$ are known.

\begin{thm} The Hankel transforms of the sequences $\seqn{\alpha^2 C_n-\beta C_{n+1}}$ and $\seqn{\alpha^2 C_{n+1}-\beta C_{n+2}}$ can be evaluated as follows
$$
\aligned
\det[\alpha^2C_{i+j}-\beta C_{i+j+1}]_{0 \leq i,j \leq n}&=\frac{1}{2^{2n+3} \sqrt{\alpha^2-4\beta}} \[(\alpha+\sqrt{\alpha^2-4\beta})^{2n+3}-(\alpha-\sqrt{\alpha^2-4\beta})^{2n+3}\] \\
\det[\alpha^2C_{i+j+1}-\beta C_{i+j+2}]_{0 \leq i,j \leq n}&=\frac{1}{2^{2n+4} \alpha \sqrt{\alpha^2-4\beta}} \[(\alpha+\sqrt{\alpha^2-4\beta})^{2n+4}-(\alpha-\sqrt{\alpha^2-4\beta})^{2n+4}\]
\endaligned
$$
\label{thm:CatLin}
\end{thm}

\begin{proof}
We use again the method based on orthogonal polynomials. It is well-known (see for example \cite{CRI}) that the Catalan numbers have the following moment representation
$$
C_n=\frac1{2\pi}\int_0^4 x^n \sqrt{\frac4{x}-1} dx.
$$
 This directly implies that the sequence $\seqn{\alpha^2 C_n-\beta C_{n+1}}$ is the moment sequence of the weight function:
$$
\hat w(x)=\frac1{2\pi}(\alpha^2-\beta x)\sqrt{\frac4{x}-1}.
$$
In order to determine the three-term recurrence relation coefficients corresponding to $\hat w(x)$, we
start from the weight function of the monic Chebyshev polynomials of the fourth kind:
$$
w^{(0)}(x)=\sqrt{\frac{1-x}{1+x}}, \quad x\in[-1,1].
$$
%Weight function of the sequence $\seqn{\alpha^2 C_n-\beta C_{n+1}}$ is $\hat w(x)=\frac1{2\pi}(\alpha^2-\beta x)\sqrt{\frac4{x}-1}$.
%We start from the monic Chebyshev polynomials of the fourth kind:
%$$
%W_n(\cos{\theta})=\frac{\sin{(n+\frac{1}{2})\theta}}{2^n\sin{\frac{\theta}2}}.
%$$
The corresponding coefficients $\alpha^{(0)}_n$ and $\beta^{(0)}_n$ are
$$
\alpha^{(0)}_0=-1/2,\qquad \alpha^{(0)}_n=0,\quad n\geq 1,\qquad \beta^{(0)}_0=\pi, \qquad \beta^{(0)}_n=1/4, \quad n\geq 1.
$$
Now we define a new weight function $w^{(1)}(x)$ by
$$
w^{(1)}(x)=w^{(0)}\(\frac{x}2-1\)
$$
and use part {\bf (2)} of Lemma \ref{lem:ModWeight1} with $a=1/2$ and $b=-1$. Hence we obtain
$$
\alpha^{(1)}_0=1, \qquad \alpha^{(1)}_n=2,\quad n\geq 1,\qquad \beta^{(1)}_0=2\pi, \qquad \beta^{(1)}_n=1, \quad n\geq 1.
$$
The next transformation is
$$
w^{(2)}(x)=-\frac{\beta}{2\pi}\cdot w^{(1)}(x).
$$
From part {\bf (1)} of Lemma \ref{lem:ModWeight1} we see that
$$
\alpha^{(2)}_0=1, \qquad \alpha^{(2)}_n=2,\quad n\geq 1,\qquad \beta^{(2)}_0=-\beta, \qquad \beta^{(2)}_n=1, \quad n\geq 1.
$$
The final transformation is a linear multiplier transformation
$$
\hat w(x)=\(x-\frac{\alpha^2}{\beta}\)\cdot w^{(2)}(x).
$$
According to Lemma \ref{lem:ModWeight1} ($c=\alpha^2/\beta$) we have to consider the following temporary sequence
\bb
r_0=\frac{\alpha^2}\beta-1, \qquad r_n=\frac{\alpha^2}\beta-2-\frac1{r_{n-1}}
\label{eq:lemdifr}
\ee
and the coefficients $\hat \beta_n$ are obtained by
$$
\hat \beta_0=\alpha^2C_0-\beta C_1=\alpha^2-\beta, \qquad \hat \beta_n=\beta^{(2)}_n \frac{r_n}{r_{n-1}}, \quad n\in \N.
$$
The Heilermann formula now yields
$$
\frac{\hat h_{n+1}}{\hat h_n}=\hat \beta_0 \hat \beta_1 \cdots \hat \beta_{n+1}=\beta r_{n+1}.
$$
Replacing the last expression into \eqref{eq:lemdifr} we obtain the following linear difference equation
\bb
\hat h_n-(\alpha^2-\beta)\hat h_{n-1}+\beta^2\hat h_{n-2}=0, \qquad (n\geq 2)
\label{eq:lemhn}
\ee
where the initial values are given by $\hat h_0=\alpha^2-\beta$ and $\hat h_1=\alpha^4-3\alpha^2\beta+\beta^2$. By solving \eqref{eq:lemhn}, we directly obtain the first statement of lemma.

\medskip

To prove the second statement, let us observe that the weight function of the sequence $\seqn{\alpha^2 C_{n+1}-\beta C_{n+2}}$ is equal to
$$
\breve w(x)=\frac1{2\pi}x(\alpha^2-\beta x)\sqrt{\frac4{x}-1}=
\frac{\beta}{\pi}\cdot \(\frac{\alpha^2}{\beta}-x\)\cdot \sqrt{1-\(\frac{x-2}{2}\)^2}.
$$
The initial weight function now is the weight function of Chebyshev polynomials of the second kind:
$$
w^{(0)}(x)=\sqrt{1-x^2}, \quad x\in[-1,1].
$$
As in the previous case, by applying the following sequence of transformations
$$
%\aligned
w^{(1)}(x)=w^{(0)}\(\frac{x-2}2\), \quad
w^{(2)}(x)=-\frac{\beta}{\pi}\cdot w^{(1)}(x), \quad
\breve w(x)=\(x-\frac{\alpha^2}{\beta}\)\cdot w^{(2)}(x)%\\
%\endaligned
$$
we prove that the sequence $\seqn{\breve h_n}$ satisfies the same linear difference equation
\bb
\breve h_n-(\alpha^2-\beta)\breve h_{n-1}+\beta^2\breve h_{n-2}=0, \qquad (n\geq 2)
\label{eq:lemhn1}
\ee
but with different initial values $\breve h_0=\alpha^2-2\beta$ and $\breve h_1=\alpha^4-3\alpha^2\beta+3\beta^2$. By solving \eqref{eq:lemhn1} we obtain the second statement of lemma.
\end{proof}

\section{The sequence $u_n^{**}$}
\label{sect:dnun**}

We can also express the sequence $u_n^{**}=u_{n+2}$ as the falling $\alpha$-binomial transformation of a certain sequence, whose Hankel determinant will be evaluated. First we need to define the generalization of the aerating transform $\A(p)$.

\begin{dfn} For a given sequence $c=\seqn{c_n}$, we define its $\alpha$-aerating transformation $b=\A(c;\alpha)$ by $a_n=\alpha p_n+p_{n+1}$, where $p=\A(c)$.
In other words, if $a=\A(c;\alpha)$ then $a=(\alpha c_0,c_1,\alpha c_1,c_2,\alpha c_2,c_3,\alpha c_3,\ldots)$.
\end{dfn}

Let $a=\A\(c;\alpha \)$, i.e. the $\alpha$-aerating transform of the sequence $c_n=\beta^n C_n$. The sequence $a=\seqn{a_n}$ can be expressed as follows
\bb
a_n=\begin{cases}
\alpha \beta^k C_k,& n=2k \\
\beta^k C_k, & n=2k-1
\end{cases}.
\label{eq:dna_n}
\ee
According to \eqref{formulaUN} we have
$$
\aligned
u^{**}_n&=\sum_{k=0}^{[\frac{n+1}{2}]} \binom {n+1}{2k}\alpha^{n+1-2k}\beta^{k}C_k\\
&=\sum_{k=0}^{[\frac{n+1}{2}]} \binom {n}{2k-1}\alpha^{n-(2k-1)}\beta^{k}C_k+\sum_{k=0}^{[\frac{n+1}{2}]} \binom {n}{2k}\alpha^{n-2k}(\alpha \beta^{k} C_k)\\
&=\sum_{l=0}^n \binom nl \alpha^{n-l} a_l.
\endaligned
$$
Hence $u^{**}=\Bf(a;\alpha)$ and from Lemma \ref{lem:Bf} we conclude that $\H(u^{**})=\H(a)$. To evaluate $\H(a)$, we need the following theorem.

From now on, we denote by $[{\bf A}]_{m\times m}$ a matrix formed by first $m$ rows and columns of the (infinite) matrix ${\bf A}$. Also we label rows and columns of matrices starting from $0$ (i.e. $0,1,2,\ldots$).

\begin{thm} Let $g=\H(a)$ and $a=\A(c;\alpha)$. Then the following holds
\label{thm:genaer}
\bb
\aligned
g_n&=\det[a_{i+j}]_{0\leq i,j\leq n}\\
&=\begin{cases}
\det[\alpha^2c_{i+j}-c_{i+j+1}]_{0 \leq i,j \leq k-1} \cdot \det[c_{i+j+1}]_{0 \leq i,j \leq k-1},& n=2k-1\\
\alpha\cdot \det[\alpha^2c_{i+j+1}-c_{i+j+2}]_{0 \leq i,j \leq k-1} \cdot \det[c_{i+j+1}]_{0 \leq i,j \leq k},& n=2k
\end{cases}.
\endaligned
\label{eq:genaer}
\ee
\end{thm}
\begin{proof} The determinant $\det[a_{i+j}]_{0\leq i,j\leq n}$  that we wish to evaluate, has the form
$$
\det[a_{i+j}]_{0\leq i,j\leq n}=
\det \bmatrix
\alpha c_0 & c_1 & \alpha c_1 & c_2 & \cdots \\
c_1 & \alpha c_1 & c_2 & \alpha c_2 & \\
\alpha c_1 & c_2 & \alpha c_2 & c_3 & \\
c_2 & \alpha c_2 & c_3 & \alpha c_3 & \\
\vdots & & & & \ddots
\endbmatrix_{(n+1)\times (n+1)}.
$$

We distinguish two cases depending on the parity of $n$.

\noindent {\bf Case 1.} $n=2k-1$ is odd. We multiply column $2j+1$ by $\alpha^{-1}$ and subtract from the column $2j$, for every $j=0,1,\ldots,k-2$. That leads to the following determinant
$$
\det[a_{i+j}]_{0\leq i,j\leq n}=
\det \bmatrix
\alpha c_0-\alpha^{-1} c_1 & c_1 & \alpha c_1-\alpha^{-1}c_2 & c_2 & \cdots \\
0 & \alpha c_1 & 0 & \alpha c_2 & \\
\alpha c_1-\alpha^{-1} c_2 & c_2 & \alpha c_2-\alpha^{-1}c_3 & c_3 & \\
0 & \alpha c_2 & 0 & \alpha c_3 & \\
\vdots & & & & \ddots
\endbmatrix_{(n+1)\times (n+1)}
$$
By exchanging rows and columns we get the block diagonal form:
$$
\det[a_{i+j}]_{0\leq i,j\leq n}=\det \bmatrix {\bf A} & * \\ & {\bf B} \endbmatrix=\det {\bf A} \cdot \det {\bf B},
$$
where star ($*$) denotes the appropriate $k\times k$ matrix which does not have an influence in determinant computation. Matrices ${\bf A}$ and ${\bf B}$ are given by
$$
{\bf A}=[\alpha c_{i+j}-\alpha^{-1}c_{i+j+1}]_{0\leq i,j \leq k-1}, \quad {\bf B}=[\alpha c_{i+j+1}]_{0\leq i,j \leq k-1}.
$$
By taking $\alpha$ from each column of matrix ${\bf B}$ and putting to the corresponding column of matrix ${\bf A}$, we get the first case of the expression \eqref{eq:genaer}.

\smallskip

\noindent {\bf Case 2.} $n=2k$ is even. Multiplying column $2j$ by $\alpha^{-1}$ and subtracting from the column $2j-1$ (for every $j=1,2,\ldots,k$) yields the determinant
$$
\det[a_{i+j}]_{0\leq i,j\leq n}=\det\bmatrix
\alpha c_0 & 0 & \alpha c_1 & 0 & \alpha c_2 & \cdots \\
c_1 & \alpha c_1-\alpha^{-1}c_2 & c_2 & \alpha c_2-\alpha^{-1}c_3 & c_3 & \\
\alpha c_1 & 0 & \alpha c_2 & 0 & \alpha c_3 \\
c_2 & \alpha c_2-\alpha^{-1}c_3 & c_3 & \alpha c_3-\alpha^{-1}c_4 & \alpha c_3 & \\
\alpha c_2 & 0 & \alpha c_3 & 0 & \alpha c_4 \\
\vdots & & & & & \ddots
\endbmatrix_{(n+1)\times (n+1)}.
$$
Again, by exchanging rows and columns we get the block diagonal form:
$$
\det[a_{i+j}]_{0\leq i,j\leq n}=\det \bmatrix {\bf A'} & * \\ & {\bf B'} \endbmatrix=\det {\bf A'} \cdot \det {\bf B'},
$$
where
$$
{\bf A'}=[\alpha c_{i+j+1}]_{0\leq i,j \leq k}, \quad {\bf B'}=[\alpha c_{i+j+1}-\alpha^{-1}c_{i+j+2}]_{0\leq i,j \leq k-1}.
$$
By taking $\alpha$ from each column of the matrix ${\bf B'}$ and putting the first $k$ entries to the corresponding columns of the matrix ${\bf A'}$, we obtain the second part of \eqref{eq:genaer}.
\end{proof}

According to the previous theorem, Proposition \ref{prop:mul} and the fact that $\H\(\seqn{C_n}\)=\seqn{1}$, for any odd $n=2k-1$, it holds that
$$
\aligned
\det[a_{i+j}]_{0\leq i,j\leq n}&=
\det[\alpha^2\beta^{i+j}C_{i+j}-\beta^{i+j+1}C_{i+j+1}]_{0 \leq i,j \leq k-1} \cdot \det[\beta^{i+j+1}C_{i+j+1}]_{0 \leq i,j \leq k-1}\\
&=\beta^{\binom n2}\det[\alpha^2C_{i+j}-\beta C_{i+j+1}]_{0 \leq i,j \leq k-1} \cdot \det[C_{i+j+1}]_{0 \leq i,j \leq k-1}\\
&=\beta^{\binom n2}\det[\alpha^2C_{i+j}-\beta C_{i+j+1}]_{0 \leq i,j \leq k-1}.
\endaligned
$$
Similarly, for even $n=2k$, we have
$$
\aligned
\det[a_{i+j}]_{0\leq i,j\leq n}&=
\alpha\det[\alpha^2\beta^{i+j+1}C_{i+j+1}-\beta^{i+j+2}C_{i+j+2}]_{0 \leq i,j \leq k-1} \cdot \det[\beta^{i+j+1}C_{i+j+1}]_{0 \leq i,j \leq k}\\
&=\alpha\beta^{\frac{n^2+2n+2}2} \det[\alpha^2C_{i+j+1}-\beta C_{i+j+2}]_{0 \leq i,j \leq k-1}.
\endaligned
$$
Now expression \eqref{eq:h_n**} (i.e. Theorem \ref{thm:h_n**}) is directly obtained using Theorem \ref{thm:CatLin}.

%Kad izmnozis redom $(\alpha+\sqrt{\alpha-4\beta})^2$ i $(\alpha+\sqrt{\alpha-4\beta})^3$ dobijas
%$$
%\aligned
%(\alpha+\sqrt{\alpha-4\beta})^2&=2(\alpha^2-2\beta+\alpha\sqrt{\alpha^2-4\beta}),\\
%(\alpha+\sqrt{\alpha-4\beta})^3&=4(\alpha^3-3\alpha\beta+(\alpha^2-\beta)\sqrt{\alpha^2-4\beta}).
%\endaligned
%$$
%U prevodu, ovaj tvoj izraz ima oblik
%$$
%\hat h_{k-1}=\frac1{2^k\sqrt{\alpha-4\beta}}\[ \frac14 (\alpha+\sqrt{\alpha-4\beta})^3\cdot \frac{(\alpha+\sqrt{\alpha-4\beta})^{2(k-1)}}{2^{k-1}}-\frac14 (\alpha-\sqrt{\alpha-4\beta})^3\cdot \frac{(\alpha-\sqrt{\alpha-4\beta})^{2(k-1)}}{2^{k-1}} \]
%$$
%Kad se ovo jos malo sredi, ispada
%$$
%\hat h_{k-1}=\frac1{2^{2k+1}\sqrt{\alpha-4\beta}}\[ (\alpha+\sqrt{\alpha-4\beta})^{2k+1}- (\alpha-\sqrt{\alpha-4\beta})^{2k+1} \]
%$$
%a to je upravo ono sto treba!!!

\section{Hankel-like determinants based on Catalan numbers}

Before we proceed to the reevaluation of the Hankel transform of the sequence $\seqn{u_n}$, we need to prove two lemmas concerning determinants which are generalizations of the Hankel determinants. Our main tool is the following theorem proven by Krattenthaller in \cite{krattCat} (Theorem 3):

\begin{thm} {\bf \cite{krattCat}} Let $n$ be a positive integer and $\alpha_0,\alpha_1,\ldots,\alpha_{n-1}$ non-negative integers. Then
\bb
\det[C_{\alpha_i+j}]_{0\leq i,j\leq k-1}=\prod_{0\leq i<j\leq k-1} (\alpha_j-\alpha_i) \prod_{i=0}^{k-1} \frac{(i+k)!(2\alpha_i)!}{(2i)!\alpha_i!(\alpha_i+k)!}.
\label{eq:Ckratt}
\ee
\label{thm:kratt}
\end{thm}

We use the notation $\chi(P)=1$ if $P$ is true and $\chi(P)=0$ otherwise. Also, we assume that the sequence $c=\seqn{c_n}$ is defined by $c_n=\beta^n C_n$.

\begin{lem} For every $l=0,1,\ldots,k-1$ we have
$$
\det[c_{i+j+\chi(j\geq l)+1}]_{0\leq i,j\leq k-1}=\beta^{k^2+k-l} \binom{l+k+1}{2l+1}.
$$
\label{lem:minori}
\end{lem}
\begin{proof} Denote by $\alpha_i=i+\chi(i\geq l)+1$. We need to evaluate the determinant
$$
\det[c_{\alpha_i+j}]_{0\leq i,j\leq k-1}=\det[\beta^{\alpha_i+j}C_{\alpha_i+j}]_{0\leq i,j\leq k-1}.
$$
By taking $\beta^{\alpha_i}$ from row $i$ ($i=0,1,\ldots,k-1$) and then $\beta^j$ from column $j$ ($j=0,1,\ldots,k-1$) we get the total power of $\beta$ equal to
$$
\sum_{i=0}^{k-1} \alpha_i + \sum_{j=0}^{k-1}j = k^2+k-l.
$$
Hence, our determinant reduces to
\bb
\det[c_{\alpha_i+j}]_{0\leq i,j\leq k-1}=\beta^{k^2+k-l}\det[C_{\alpha_i+j}]_{0\leq i,j\leq k-1}.
\label{eq:lemminori1}
\ee
According to the Theorem \ref{thm:kratt}, we need to compute the following products
$$
 P_1=\prod_{0\leq i<j\leq k-1} (\alpha_j-\alpha_i), \quad P_2=\prod_{i=0}^{k-1} \frac{(i+k)!(2\alpha_i)!}{(2i)!\alpha_i!(\alpha_i+k)!}.
$$
By direct evaluation it can be shown that
\bb
P_1=\binom kl \prod_{j=0}^{k-1} j!, \quad P_2=\frac{k!(l+1)!}{\displaystyle\prod_{j=0}^{k+1} j!} \cdot \frac{(2k+2)(l+k+1)!}{(2l+2)!}.
\label{eq:lemminori2}
\ee
The first product was evaluated considering three different cases ($i<j<l$, $l\leq i<j$, $i\leq l<j$), while for the second we only needed to distinguish between $i<l$ and $i\geq l$. Now using \eqref{eq:lemminori2} and Theorem \ref{thm:kratt} we obtain
$$
\det[C_{\alpha_i+j}]_{0\leq i,j\leq k-1}=P_1 \cdot P_2=\binom{l+k+1}{2l+1}.
$$
Now the statement of the lemma follows directly from the previous equation and \eqref{eq:lemminori1}.
\end{proof}

\begin{lem} For every $l=0,1,\ldots,k-1$ we have
$$
\det[c_{i+j+\chi(j\geq l)}]_{0\leq i,j\leq k-1}=\beta^{k^2-l} \binom{l+k}{2l}.
$$
\label{lem:minori1}
\end{lem}
\begin{proof} We use again the same procedure. Now we denote $\alpha_i=i+\chi(i\geq l)$ and obtain
$$
\det[c_{\alpha_i+j}]_{0\leq i,j\leq k-1}=\det[\beta^{\alpha_i+j}C_{\alpha_i+j}]_{0\leq i,j\leq k-1}=\beta^{k^2-l} \det[C_{\alpha_i+j}]_{0\leq i,j\leq k-1}.
$$
The product $P_1$ has the same value as in the previous case, while $P_2$ is equal to
$$
P_2=\frac{(l+1)!}{\displaystyle\prod_{j=0}^k j!} \cdot \frac{(l+k)!}{(2l)!}
$$
Again, by replacing $\det[C_{\alpha_i+j}]_{0\leq i,j\leq k-1}=P_1\cdot P_2$ (Theorem \ref{thm:kratt}) we obtain the statement of the lemma.
\end{proof}

\section{The sequence $u_n$}

In the section \ref{sect:dnun**} we proved that $u^{**}=\Bf(a;\alpha)$, where $a$ is the $\alpha$-aerating transform of the sequence $c_n=\beta^nC_n$ ($a=\A(c;\alpha)$), i.e. (equation \eqref{eq:dna_n}):
$$
a_n=\begin{cases}
\alpha \beta^k C_k,& n=2k \\
\beta^k C_k, & n=2k-1
\end{cases}.
$$
According to Lemma \ref{lem:binmatrix}, we have $\HH_{u^{**}}=\BB^\alpha \HH_a (\BB^\alpha)^T$. We have already proved that $u^*=\Bf\(p;\alpha\)$ (section \ref{sect:dnun*}) where $p=\A(c)$, i.e.
$$
p_n=\begin{cases}
\beta^k C_k, & n=2k \\
0, & n=2k-1
\end{cases}.
$$
This can be written in matrix notation as $u^*=\BB^\alpha p$. Now we have that the following matrix equality holds:
\bb
\bmatrix 1 & \\ & \BB^\alpha \endbmatrix
\bmatrix 0 & p^T \\ p & \HH_a \endbmatrix
\bmatrix 1 & \\ & (\BB^\alpha)^T \endbmatrix=
\bmatrix 0 & p^T (\BB^\alpha)^T \\ \BB^\alpha p & \BB^\alpha \HH_a (\BB^\alpha)^T \endbmatrix=
\bmatrix 0 & (u^*)^T \\ u^* & \HH_{u^{**}} \endbmatrix = \HH_u
\ee
Hence, the determinant of the $(n+1)\times (n+1)$ principal minor of $\HH_u$, formed by the rows and columns with indices $1,2,\ldots,n+1$, is equal to the same minor of the matrix
$$
\HH'=\bmatrix 0 & p^T \\ p & \HH_a \endbmatrix.
$$
That minor is exactly $h_n$, i.e. $n$-th member of the Hankel transform $h=\H(u)$. In other words, we have to compute
$$
h_n=\det [\HH_u]_{(n+1)\times(n+1)} = \det \bmatrix 0 & p^T \\ p & \HH_a \endbmatrix_{(n+1)\times(n+1)} = \det
\bmatrix
0 & c_0 & 0 & c_1 & 0 & \cdots \\
c_0 & \alpha c_0 & c_1 & \alpha c_1 & c_2 & \\
0 & c_1 & \alpha c_1 & c_2 & \alpha c_2 & \\
c_1 & \alpha c_1 & c_2 & \alpha c_2 & c_3 & \\
0 & c_2 & \alpha  c_2 & c_3 & \alpha  c_3 & \\
\vdots & & & & & \ddots
\endbmatrix_{(n+1)\times(n+1)}.
$$
We distinguish two cases depending on the parity of $n$.

\medskip

\noindent {\bf Case 1.} $n=2k-1$ is odd. Multiplying the column $2j$ by $\alpha$ and subtracting from column $2j+1$ ($j=0,1,\ldots,k-1$) yields
$$
h_n=\det
\bmatrix
0 & c_0 & 0 & c_1 & 0 & \cdots & \\
c_0 & 0 & c_1 & 0 & c_2 &\\
0 & c_1 & \alpha c_1 & c_2-\alpha^2 c_1 & \alpha c_2   \\
c_1 & 0 & c_2 & 0 & c_3  \\
0 & c_2 & \alpha c_2 & c_3-\alpha^2 c_2 & \alpha c_3 \\\
\vdots & & & & & \ddots\\
\endbmatrix_{(n+1)\times(n+1)}
$$
By exchanging rows and columns in the previous determinant we obtain
\bb
h_n=(-1)^k\det
\bmatrix
{\bf A} & * \\ & {\bf B}
\endbmatrix=(-1)^k\det {\bf A} \cdot \det {\bf B}
\label{eq:Hfact}
\ee
where the matrices ${\bf A}$ and ${\bf B}$ are equal to
$$
{\bf A}=\det \bmatrix
c_0 & c_1 & \cdots & c_{k-1}\\
c_1 & c_2-\alpha^2 c_1 & & c_{k-2}-\alpha^2 c_{k-1}\\
\vdots & & \ddots\\
c_{k-1} & c_{k-2}-\alpha^2 c_{k-1} & & c_{2k-2}-\alpha^2 c_{2k-3}
\endbmatrix, \quad
{\bf B}=[c_{i+j}]_{0\leq i,j \leq k-1}.
$$
Since $c_n=\beta^nC_n$, using Proposition \ref{prop:mul} we obtain
\bb
\det {\bf B}=\det[\beta^{i+j}C_{i+j}]_{0\leq i,j \leq k-1}=\beta^{k(k-1)}\det [C_{i+j}]_{0\leq i,j\leq k-1}=\beta^{k(k-1)}.
\label{eq:detB}
\ee
By adding column $j$ to column $j+1$ of the matrix ${\bf A}$ ($j=0,1,\ldots,k-2$) we obtain the following determinant
$$
\det {\bf A}=\det
\bmatrix
c_0 & \alpha^2 c_0+c_1 & \alpha^4 c_0+\alpha^2 c_1+c_2 & \cdots\\
c_1 & c_2 & c_3 \\
c_2 & c_3 & c_4 \\
\vdots & & & \ddots
\endbmatrix_{k \times k}.
$$
Expanding over the first row yields
\bb
\det {\bf A}=\sum_{l=0}^{k-1} (-1)^l \(\sum_{h=0}^l \alpha^{2h} c_{l-h} \) \det[c_{i+j+\chi(j\geq l)+1}]_{0\leq i,j\leq k-2}.
\label{eq:Aexpand}
\ee
Using Lemma \ref{lem:minori} together with the expressions \eqref{eq:Hfact}, \eqref{eq:detB} and \eqref{eq:Aexpand}, we obtain
\bb
h_{2k-1}=\beta^{(k-1)(2k-1)} \sum_{l=0}^{k-1} (-1)^{k+l} \(\sum_{h=0}^l \alpha^{2h} \beta^{k-1-h} C_{l-h} \) \binom{l+k}{2l+1}.
\label{eq:hodd}
\ee

\medskip

\noindent {\bf Case 2.} $n=2k$ is even. Subtracting column $2j-1$ from column $2j$ ($j=0,1,\ldots,k-1$) we obtain
$$
h_n=\det
\bmatrix
0 & c_0 & -\alpha c_0 & c_1 & 0 & \cdots & \\
c_0 & \alpha c_0 & c_1-\alpha^2c_0 & \alpha c_1 & c_2-\alpha^2c_1 &\\
0 & c_1 & 0 & c_2 & 0   \\
c_1 & \alpha c_1 & c_2-\alpha^2c_1 & \alpha c_2 & c_3-\alpha^2c_2  \\
0 & c_2 & 0 & c_3 & 0 \\\
\vdots & & & & & \ddots\\
\endbmatrix_{(n+1)\times(n+1)}
$$
By exchanging rows and columns in the previous determinant we obtain
\bb
h_n=(-1)^{k}\det
\bmatrix
{\bf A} & * \\ & {\bf B}
\endbmatrix=(-1)^k\det {\bf A} \cdot \det {\bf B}
\label{eq:Hfact1}
\ee
where the matrices ${\bf A}$ and ${\bf B}$ are equal to
$$
{\bf A}=\det \bmatrix
0  & -\alpha c_0 & \cdots & -\alpha  c_{k-1}\\
c_0 & c_1-\alpha^2 c_0 & & c_{k}-\alpha^2 c_{k-1}\\
\vdots & & \ddots\\
c_{k-1} & c_{k}-\alpha^2 c_{k-1} & & c_{2k-1}-\alpha^2 c_{2k-2}
\endbmatrix, \quad
{\bf B}=[c_{i+j+1}]_{0\leq i,j \leq k-1}.
$$
Since $c_n=\beta^nC_n$, using Proposition \ref{prop:mul} we obtain
\bb
\det {\bf B}=\det[\beta^{i+j+1}C_{i+j+1}]_{0\leq i,j \leq k-1}=\beta^{k^2}\det [C_{i+j}]_{0\leq i,j\leq k-1}=\beta^{k^2}.
\label{eq:detB1}
\ee
Again, by multiplying column $j$ by $\alpha$ and adding to column $j+1$ ($j=0,1,\ldots,k-1$), we obtain the following determinant
$$
\det {\bf A}=\det
\bmatrix
0 & -\alpha c_0 & \alpha^3 c_0-\alpha c_1 & -\alpha^5 c_0-\alpha^3 c_1-\alpha c_2 & \cdots\\
c_0 & c_1 & c_2 & c_3 \\
c_1 & c_2 & c_3 & c_4 \\
c_2 & c_3 & c_4 & c_5 \\
\vdots & & & & \ddots
\endbmatrix_{(k+1) \times (k+1)}.
$$
which can be expanded by the first row in the following way
\bb
\det A=\sum_{l=1}^{k} (-1)^l \(\sum_{h=0}^{l-1} \alpha^{2h+1} c_{l-1-h} \) \det[c_{i+j+\chi(j\geq l)}]_{0\leq i,j\leq k-1}.
\label{eq:Aexpand1}
\ee
Now using Lemma \ref{lem:minori1} and expressions \eqref{eq:Hfact1}, \eqref{eq:detB1} and \eqref{eq:Aexpand1} we obtain
\bb
h_{2k}=\beta^{k(2k-1)} \sum_{l=1}^{k} (-1)^{k+l-1} \(\sum_{h=0}^{l-1} \alpha^{2h+1} \beta^{k-1-h} C_{l-1-h} \) \binom{l+k}{2l}.
\label{eq:heven}
\ee

\noindent{\bf Proof of Theorem \ref{thm:h_n}.} We can rewrite expressions \eqref{eq:hodd} and \eqref{eq:heven} as follows (we exchanged the order of summation):
\bb
\aligned
h_{2k-1}&=\beta^{(k-1)(2k-1)} \sum_{h=0}^{k-1} \alpha^{2h} \beta^{k-1-h} \sum_{l=h}^{k-1} (-1)^{k+l} C_{l-h} \binom{l+k}{2l+1}, \\
h_{2k}&=\beta^{k(2k-1)} \sum_{h=0}^{k-1} \alpha^{2h+1} \beta^{k-1-h} \sum_{l=h+1}^{k} (-1)^{k+l-1} C_{l-1-h} \binom{l+k}{2l}.
\label{eq:hoddeven}
\endaligned
\ee
Now let $z_n=\beta^{-\binom n2}h_n$ and in the second equation of \eqref{eq:hoddeven} decrease the bounds for $l$ by 1. Expressions for $z_{2k}$ and $z_{2k-1}$ are
\bb
\aligned
z_{2k-1}&=\sum_{h=0}^{k-1} \alpha^{2h} \beta^{k-1-h} \sum_{l=h}^{k-1} (-1)^{k+l} C_{l-h} \binom{l+k}{2l+1}, \\
z_{2k}&=  \sum_{h=0}^{k-1} \alpha^{2h+1} \beta^{k-1-h} \sum_{l=h}^{k-1} (-1)^{k+l} C_{l-h} \binom{l+k+1}{2l+2}.
\endaligned
\ee
By direct verification we conclude that $z_n$ satisfies the three-term linear difference equation
$$
z_{n+2}-\alpha z_{n+1}+\beta z_n=0
$$
for all $n\in \N_0$, which directly implies the expression \eqref{formzahn}:
$$
h_n=\beta^{\binom n2}z_n=\frac{\beta^{\binom{n}2}}{2^n \sqrt{\alpha^2-4\beta}} \[\(\alpha-\sqrt{\alpha^2-4\beta}\)^n-\(\alpha+\sqrt{\alpha^2-4\beta}\)^n \].
$$
This completes the proof of Theorem \ref{thm:h_n}.

\section*{Acknowledgements}

Marko D. Petkovi\' c gratefully acknowledges the support of the research project 174013 of the Serbian Ministry of Education and Science. Authors wish to thank anonymous reviewer of our previous paper \cite{part1} whose comments were the motivation for this work. Also the authors wish to thank Professor Predrag M. Rajkovi\'c for useful discussions on this topic.

%\smallskip

%\noindent Authors would like to thank to anonymous referee for useful comments and remarks which improved the quality of the paper. Those comments were main motivation for the results shown in section 6.

\end{document}